\documentclass[USenglish,square,sort,comma,numbers]{article}

\usepackage[utf8]{inputenc}

\usepackage{amsmath,amsfonts,amssymb,amsthm,comment,cases,accents}

\newcommand{\N}{{\mathbb N}}
\newcommand{\R}{{\mathbb R}}

\theoremstyle{plain}
\newtheorem{Theorem}{Theorem}[section]

\newtheorem{lemma}[Theorem]{Lemma}
\newtheorem{Proposition}[Theorem]{Proposition}

\newenvironment{proposition}{\begin{Proposition} }{\end{Proposition}}
\newenvironment{Proof}{\begin{proof} }{\end{proof}}
\newenvironment{theorem}{\begin{Theorem} }{\end{Theorem}}

\theoremstyle{definition}
\newtheorem{Remark}[Theorem]{Remark}
\newenvironment{remark}{\begin{Remark} \rm}{\end{Remark}}

\def\dist{{\rm dist}\, }

\def\al{{\alpha}}
\def\D{{\Delta}}
\def\eps{{\varepsilon}}
\def\p{{\varphi}}
\def\o{{\omega}}
\def\O{{\Omega}}

\def\n{{\mathbf n}}
\def\bfA{{\mathbf A}}
\def\bfE{{\mathbf E}}
\def\bfH{{\mathbf H}}

\newcommand{\oveta}{{\overline \eta}_u}
\newcommand{\ovxi}{{\overline \xi}_u}

\begin{document}


\title{Klein-Gordon-Maxwell Systems with Nonconstant Coupling Coefficient}
\author{Monica Lazzo \thanks{
 Dipartimento di Matematica, Universit\`a degli Studi di Bari Aldo Moro, via E.~Orabona 4, 70125 Bari, Italy; e-mail: monica.lazzo@uniba.it,
 lorenzo.pisani@uniba.it} \and Lorenzo Pisani \footnotemark[1]}
 \date{}
\maketitle

  \abstract{We study a Klein-Gordon-Maxwell system, in a bounded spatial domain, under Neumann boundary conditions on the electric potential. We allow a nonconstant coupling coefficient. For sufficiently small data, we find infinitely many static solutions. \\[2mm]
  }

\section{Introduction}\label{intro}
We are interested in the system of nonautonomous elliptic equations
\begin{equation}\label{SYS-no}
\begin{alignedat}{2}
\Delta u  &= m^2 \, u - \bigl(q(x)\, \phi \bigr)^2 \, u  \qquad & \hbox{in $\O$,}  \\[1mm]
\Delta \phi &= \bigl(q(x)\, u \bigr)^2 \, \phi & \hbox{in $\O$,}
\end{alignedat}
\end{equation}
where $\D$ is the Laplace operator in $\R^3$, $\O \subset \R^3$ is a bounded and  smooth domain, $m \in \R$,  $q \in L^6(\O) \setminus\{0\}$. We complement these equations with the boundary conditions
\begin{subequations}\label{BC-no}
\begin{alignat}{2}
u  &= 0  \qquad & \hbox{on $\partial \O$,} \label{BC-a} \\
\dfrac{\partial \phi}{\partial \n}  &= \al & \hbox{on $\partial \O$,}
\label{BC-b}
\end{alignat}
\end{subequations}
where $\n$ is the unit outward normal vector to $\partial \O$ and $\al \in H^{1/2}(\partial \O)$.

\goodbreak
We look for {\em nontrivial} solutions, by which we mean pairs $(u,\phi) \in H^1_0(\O) \times H^1(\O)$, satisfying~\eqref{SYS-no}-\eqref{BC-no} in the usual weak sense, with $u\ne0$. Note that, if $(u,\phi)$ is a nontrivial solution, the pair $(-u,\phi)$ is a nontrivial solution  as well.

System~\eqref{SYS-no} arises in connection with the so-called Klein-Gordon-Maxwell equations, which model the interaction of a charged matter field with the electromagnetic field $(\bfE,\bfH)$.
They are the Euler-Lagrange equations of the Lagrangian density
\begin{equation*}
\begin{split}
{\cal L}_{KGM} & = \tfrac{1}{2} \bigl( |(\partial_t + i \, q \, \phi) \, \psi|^2 - |(\nabla - i \, q \, \bfA) \, \psi|^2 - m^2 \, |\psi|^2 \bigr) + {} \\[2mm]
& + \tfrac{1}{8\pi} \bigl(|
\nabla \phi + \partial_t \bfA|^2 - |\nabla \times \bfA|^2 \bigr) \, ,
\end{split}
\end{equation*}
where $\psi$ is a complex-valued function representing the matter field, while $\phi$ and $\bfA$ are the gauge potentials, related to the electromagnetic field via the equations
$\bfE = -\nabla \phi - \partial_t \bfA$, $\bfH = \nabla \times \bfA$.
For the derivation of the Lagrangian density, and details on the physical model, we refer to~\cite{BF-2002, bleeker, felsager}.  Let us point out that, in the physical model, $q$ is a constant which represents the electric charge of the matter field;  nonconstant coupling coefficients, however, are worth investigating from a mathematical viewpoint.

Confining attention to {\em standing waves}, in equilibrium with a purely electrostatic field, amounts to imposing $\psi(t,x) = e^{i \o t} \, u(x)$, where $u$ is a real-valued function and $\o$ is a real number, $\bfA = 0$, and $\phi=\phi(x)$. With  these choices, the Klein-Gordon-Maxwell equations considerably simplify and become
\begin{equation}\label{KGM}
\begin{alignedat}{2}
\Delta u  &= m^2 \, u - \bigl(\o + q(x) \, \phi\bigr)^2 u   \qquad & \hbox{in $\O$,} \\[1mm]
\Delta \phi &= q(x) \, \bigl(\o+ q(x) \, \phi\bigr) \, u^2 & \hbox{in $\O$.} \end{alignedat}
\end{equation}
In the special case of {\em static solutions}, corresponding to $\o=0$, System~\eqref{KGM} reduces to~\eqref{SYS-no}. In the physical model, the boundary condition~\eqref{BC-a} means that the matter field is confined to the region $\O$, while~\eqref{BC-b} amounts to prescribing the normal component of the electric field on $\partial \O$; up to a sign, the surface integral $\int_{\partial \O} \al \, d\sigma$ represents the flux of the electric field through the boundary of $\O$, and thus, the total charge contained in $\O$.

Problem~\eqref{KGM}-\eqref{BC-no} was investigated in~\cite{DPS-2010}, for a constant coupling coefficient $q$. In this case, a degeneracy phenomenon occurs and the existence of solutions to~\eqref{KGM}-\eqref{BC-no} does not depend on $\o$ (see~\cite{DPS-2010} and~\cite[Remark~1.2]{DPS-2014}). Thus, for autonomous systems, letting $\o=0$ in~\eqref{KGM} entails no loss of generality; this is not true if the coupling coefficient is not constant. The existence of infinitely many static solutions in the nonautonomous case was proved in~\cite{DPS-2014}, under the assumption that $q$ vanishes at most on a set of measure zero. Our main result generalizes Theorem~1.3 in~\cite{DPS-2014}, in that we impose no conditions on the zero-level set of $q$, and provides additional information on the solutions. We will address Problem~\eqref{KGM}-\eqref{BC-no} with $\o\ne 0$ in a forthcoming paper.

\goodbreak
\begin{theorem}\label{main}
Assume $\int_{\partial \O} \al \, d\sigma \ne 0$. There exists $\delta \in (0,\infty)$ such that, if $\|q\|_{L^6(\O)}\, \|\al\|_{H^{1/2}(\partial \O)}   < \delta$, the problem~\eqref{SYS-no}-\eqref{BC-no} has a sequence $\{(u_n,\phi_n)\}$ of nontrivial solutions with the following properties:
\begin{enumerate}
\item[{\rm (i)}] $u_0 \ge 0$ in $\O$;
\item[{\rm (ii)}] every bounded subsequence $\{u_{k_n}\}$ satisfies $\|q \, u_{k_n}\|_{L^3(\O)} \to 0$ as $n$ goes to infinity.
\end{enumerate}
\end{theorem}

\begin{remark}
Unless $\|u_n\|_{H^1_0(\O)}\to \infty$ as $n\to \infty$, bounded subsequences of the sequence $\{u_n\}$ do exist. Plainly, any such subsequence has in $L^6(\O)$ a limit point $u$  such that $q \, u = 0$.
\end{remark}

At least for small data, assuming $\int_{\partial \O} \al \, d\sigma \ne 0$ is necessary for the existence of nontrivial solutions, as the following result shows.

\begin{theorem}\label{A-zero}
Suppose $\int_{\partial \O} \al \, d\sigma=0$. With the same $\delta$ as in Theorem~\ref{main}, assume $\|q\|_{L^6(\O)} \, \|\al\|_{H^{1/2}(\partial \O)} < \delta$. Then: Problem~\eqref{SYS-no}-\eqref{BC-no} has no  nontrivial solutions.
\end{theorem}

Note that, if $\int_{\partial \O} \al \, d\sigma=0$, every pair $(0,\phi)$, with $\phi$  harmonic function satisfying the Neumann boundary condition~\eqref{BC-b}, is a trivial solution to~\eqref{SYS-no}-\eqref{BC-no}.

Our results are obtained by way of variational methods.
We follow an approach introduced by Benci and Fortunato (in~\cite{BF-2002} for Klein-Gordon-Maxwell systems, and earlier in~\cite{BF-1998} for Schr\"odinger-Maxwell systems) and subsequently implemented by many authors. Most results in the literature concern systems posed in unbounded spatial domains, possibly featuring lower-order nonlinear perturbations;  for instance, see~\cite{azzo-pomp, carriao, daprile-mugnai, li, xu}. We also refer to~\cite{ghimenti1, ghimenti2} for  recent applications to Klein-Gordon-Maxwell systems with Neumann boundary conditions on Riemannian manifolds.

To prove our multiplicity result, we apply Ljusternik-Schnirelmann theory to a functional $J$, defined in a subset of $H^1_0(\O)$, whose critical points correspond with nontrivial solutions to Problem~\eqref{SYS-no}-\eqref{BC-no}. The definition of  $J$ depends on whether a certain Neumann problem is uniquely solvable. The easiest way to guarantee that this occurs is to assume, as in~\cite{DPS-2014}, that $q$ vanishes at most on a set of measure zero. Here, instead, we build the solvability requirement into $\Lambda_q$, the domain of $J$.

The paper is organized as follows. In Section~\ref{prelims} we define the set $\Lambda_q$ and address the solvability issue. In Section~\ref{J-birth} we define the functional $J$ and investigate its properties. Section~\ref{proof-of-main} is devoted to the proofs of Theorems~\ref{main} and~\ref{A-zero}.

\section{Preliminaries}\label{prelims}

Throughout the paper we will use the following notation:
\begin{itemize}
\item For any integrable function $f:\O\to \R$, $\|f\|_p$ is the usual norm in $L^p(\O)$ ($p\in[1,\infty]$) and $\overline f$ is the average of $f$ in $\O$;
\item $H^1_0(\O)$ is endowed with the norm $\|\nabla f\|_2$;
\item $H^1(\O)$ is endowed with the norm $\|f\|:= \left(\|\nabla f\|_2^2 + \left|\overline f\right|^2\right)^{1/2}$;
\item $H':= L(H^1(\O),\R)$;
\item $A:= \int_{\partial \O} \al \, d\sigma$, \ $\|\al\|_{1/2} := \|\al\|_{H^{1/2}(\partial \O)}$.
\end{itemize}

\subsection{Reduction to homogeneous boundary conditions}

We begin by turning Problem \eqref{SYS-no}-\eqref{BC-no} into an equivalent problem with homogeneous boundary conditions in both variables.
Let $\chi \in H^2(\O)$ be the unique solution of
\begin{equation}\label{chi}
\Delta \chi = \dfrac{A}{|\O|} \quad \hbox{in $\O$}\, , \qquad
 \dfrac{\partial \chi}{\partial \n} = \al \quad \hbox{on $\partial \O$}\, , \qquad
\int_\O \chi \, dx = 0 \, .
\end{equation}
With $\varphi:= \phi-\chi$, Problem~\eqref{SYS-no}-\eqref{BC-no} is equivalent to
\begin{equation}\label{PROB}
\begin{cases}
\ \Delta u = m^2 u - q^2 \, (\p+\chi)^2 \, u \quad & \hbox{in $\O$,} \\[1mm]
\ \Delta \p =  (q\, u)^2 \, (\p+\chi) -  \dfrac{A}{|\O|} & \hbox{in $\O$,} \\[1mm]
\ \ u = \dfrac{\partial \p}{\partial \n} = 0 & \hbox{on $\partial \O$.}
\end{cases}
\end{equation}

Weak solutions of~\eqref{PROB} correspond with critical points of
the functional $F$ defined in $H^1_0(\O) \times H^1(\O)$ by
\begin{equation*}
F(u,\p) =   \|\nabla u\|_2^2 +  \int_\O \bigl(m^2 - q^2\, (\p+\chi)^2 \bigr) u^2 \, dx - \|\nabla \p\|_2^2  + 2\, A \, \overline \p \, .
\end{equation*}
Indeed, standard computations show that $F$ is continuously differentiable in $H^1_0(\O) \times H^1(\Omega)$ with
\begin{align*}
\langle F'_u(u,\p) , v \rangle  &= 2 \int_\O \Bigl( \nabla u \nabla v  + \bigl(m^2 - q^2\, (\p+\chi)^2 \bigr) \, u \, v \Bigr) \, dx \, , \\[1mm]
\langle F'_\p(u,\p) , \psi \rangle  & =
- 2 \int_\O \Bigl( \nabla \p \nabla \psi  + \Bigl( (q\, u)^2\, (\p+\chi)  - \dfrac{A}{|\O|} \Bigr) \, \psi \Bigr)  \, dx \, ,
\end{align*}
for every $u, v \in  H^1_0(\O)$ and $\p, \psi\in H^1(\O)$.
However, $F$ is unbounded from above and from below, even modulo compact perturbations; this precludes a straightforward application of classical results in critical point theory.

Following~\cite{BF-2002}, we associate solutions to Problem~\eqref{PROB} with critical points of a functional~$J$ that depends only on the variable $u$ and falls within the scope of classical critical point theory. Roughly speaking, $J$ is the restriction of $F$ to the zero-level set of $F'_\p$. A key ingredient in the construction of $J$ is the invertibility of the map defined in the following proposition.

\begin{proposition}\label{lax-mil}
For $b \in L^3(\O)$, define ${\cal A}_b : H^1(\O) \to H'$ by
\begin{equation*}
\langle {\cal A}_b(\p),\psi \rangle := \int_\O \bigl( \nabla \p \nabla \psi + b^2  \p \, \psi \bigr)\, dx \, ; \end{equation*}
let $\displaystyle c_b := \inf_{\|\p\|=1} \langle{\cal A}_b(\p),\p\rangle$.
\begin{itemize}
\item[{\rm (a)}]
The map $b \in L^3(\O) \mapsto {\cal A}_b \in L(H^1(\O);H')$ is continuous.
\item[{\rm (b)}]
Assume $b\ne0$. Then: $c_b >0$, the map ${\cal A}_b$ is an isomorphism, and ${\cal L}_b:={\cal A}_b^{-1}$ has continuity constant $1/c_{b}$.
\item[{\rm (c)}]
The map $b \in L^3(\O)\setminus\{0\} \mapsto {\cal L}_b \in L(H';H^1(\O))$
is continuous.
\end{itemize}
\end{proposition}

\begin{proof}
(a)\ Let $b_n , b\in L^3(\O)$. Suppose $\|b_n -b\|_3 \to 0$, hence $\|b_n^2 -b^2\|_{3/2} \to 0$.
By H\"older's inequality and Sobolev's embedding theorem, for every $n$ and for every $\p, \psi \in H^1(\O)$, we have
\begin{equation*}
\begin{split}
\bigl| \langle ({\cal A}_{b_n} - {\cal A}_b)(\p) , \psi \rangle \bigr| & =
\biggl| \int_\O (b_n^2 - b^2) \, \p\, \psi \, dx  \biggr|
\le c \, \|b_n^2-b^2\|_{3/2} \; \|\p\| \, \|\psi\|
\end{split}
\end{equation*}
for some $c\in(0,\infty)$. This implies ${\cal A}_{b_n} \to {\cal A}_b$ in $L(H^1(\O);H')$. \\
(b)\ Let $b \in L^3(\O) \setminus\{0\}$. By way of contradiction, suppose $c_b =0$ and take a sequence $\{\p_n\} \subset H^1(\O)$ such that $\|\p_n\| = 1$ and $\langle {\cal A}_b(\p_n),\p_n \rangle  \to 0$. Since
$\langle {\cal A}_b(\p_n),\p_n \rangle  \ge \|\nabla \p_n\|_2^2$, we get
$\|\nabla \p_n\|_2 \to 0$, which implies $\|\p_n - \overline \p_n\|_p\to 0$ for every $p \in [1,6]$ (by the Poincar\'e-Wirtinger inequality) and $|\overline \p_n|\to 1$.
Now observe that
\begin{equation}\label{quad}
\int_\O b^2 \, \p_n^2 \, dx  =
\int_\O \Bigl( b^2 \, (\p_n-\overline \p_n)^2 + 2 \, b^2 \, (\p_n-\overline \p_n) \, \overline \p_n + b^2 \, \overline \p_n^2 \Bigr) \, dx \, .
\end{equation}
Being smaller than $\langle {\cal A}_b(\p_n),\p_n \rangle$, the left-hand side
in~\eqref{quad} tends to $0$; moreover,
\begin{gather*}
\int_\O b^2 \, (\p_n-\overline \p_n)^2 \, dx \le \|b\|_3^2 \, \|\p_n-\overline \p_n\|_6^2 \to 0 \, , \\
\biggl|\int_\O b^2 \, (\p_n-\overline \p_n) \, \overline \p_n \, dx\biggr| \le
|\overline \p_n| \, \|b\|_3^2 \, \|\p_n-\overline \p_n\|_3 \to 0 \, ,
\\ \int_\O b^2 \, \overline \p_n^2 \, dx  \to  \|b\|_2^2 \, .
\end{gather*}
Thus, \eqref{quad} yields $b=0$, a contradiction.
The remaining assertions follow from the Lax-Milgram lemma, which is applicable because the bilinear form associated with ${\cal A}_b$ is coercive, with coercivity constant $c_b$. \\
(c)\ The assertion readily follows from Part (a) and the continuity of the inversion operator.
\end{proof}

\begin{remark}\label{equation}
For $\rho\in L^{6/5}(\O)$, let ${\cal T}_\rho$ be the linear form defined by $\langle {\cal T}_{\rho} ,\p \rangle :=\int_\O \rho\, \p \, dx$. Following common practice, we will sometimes identify ${\cal T}_\rho$ with $\rho$. \\
Fix $b \in L^3(\O) \setminus \{0\}$. In view of Proposition~\ref{lax-mil},
${\cal L}_{b}(\rho)$ is the unique solution in $H^1(\O)$ of the homogeneous Neumann problem associated with the equation
\begin{equation*}
- \Delta \p + b^2 \, \p = \rho .
\end{equation*}
Note that $\|{\cal L}_{b}(\rho)\| \le  {\|\rho\|_{6/5}}/{c_b}$.
Furthermore, ${\cal L}_{b}(\rho)$ depends continuously on $b$ and $\rho$: if  $b_n \to b$ in $L^3(\O)\setminus \{0\}$ and $\rho_n \to \rho$ in $L^{6/5}(\O)$, then
\begin{equation*}
\|{\cal L}_{b_n} (\rho_n) - {\cal L}_{b} (\rho)\| \le
\|{\cal L}_{b_n}\| \, \|\rho_n - \rho\|_{6/5} + \|{\cal L}_{b_n} - {\cal L}_{ b}\| \, \|\rho\|_{6/5} \to 0 \, .
\end{equation*}
\end{remark}

\begin{remark}\label{max-prin}
With the same notation as in the previous remark, suppose that $\rho$ does not change sign in $\O$.
Since the bilinear form associated with ${\cal A}_b$ is symmetric, ${\cal L}_{b}(\rho)$ can be characterized as the unique minimizer of the functional $f: H^1(\O) \to \R$ defined by $f(\p) = \frac{1}{2} \, \langle {\cal A}_b(\p),\p \rangle - \langle {\cal T}_\rho, \p \rangle$.
Observing that $f({\rm sign}(\rho) |{\cal L}_{b}(\rho)|) \le f({\cal L}_{b}(\rho))$, we obtain ${\rm sign}(\rho) |{\cal L}_{b}(\rho)| = {\cal L}_{b}(\rho)$, which implies $\rho \, {\cal L}_{b}(\rho) \ge 0$ in $\O$.
\end{remark}

\subsection{The set $\Lambda_q$}\label{lambda}

For $u \in H^1_0(\O)$, let $\rho_u := \dfrac{A}{|\O|} - (q \, u)^2 \, \chi$.
With the notation introduced in Proposition~\ref{lax-mil},
we have
\begin{equation*}
F'_\p(u,\p) = 2 \, \bigl( -  {\cal A}_{qu} (\p) + {\rho_u} \bigr)
\end{equation*}
for every $(u,\p)\in \Lambda_q \times H^1(\O)$.
By Proposition~\ref{lax-mil}(b), the operator ${\cal A}_{qu}$ is invertible if, and only if, $u$ belongs to the set
\begin{equation*}
\Lambda_q := \Bigl\{ u \in H^1_0(\O) \, \bigl| \, q \, u \not= 0 \Bigr\} \, .
\end{equation*}

Incidentally, we point out that, in order to find nontrivial solutions to~\eqref{PROB}, confining~$u$ within~$\Lambda_q$ is not a mere technical requirement. Indeed, if $(u,\p)$ is a solution to~\eqref{PROB} and $q \, u =0$, then $u$ satisfies $\Delta u  = m^2 \, u$ in $\O$, hence $u =0$.

If $q$ vanishes at most on a set of measure zero, as assumed in~\cite{DPS-2014}, $\Lambda_q$ equals $H^1_0(\O) \setminus\{0\}$. In general, $\Lambda_q$ satisfies the following properties.

\begin{proposition}\label{lambda-q}~{}
\begin{enumerate}
\item[{\rm (a)}] $\Lambda_q$ is open in $H^1_0(\O)$ with $\partial \Lambda_q = \bigl\{ u \in H^1_0(\O) \, | \, q \, u = 0 \bigr\}$.
\item[{\rm (b)}] If $u\in H^1_0(\O)$ and $\dist(u,\partial \Lambda_q) \to 0$, then $\|q \, u\|_3 \to 0$.
\item[{\rm (c)}] $\Lambda_q$ contains subsets with arbitrarily large genus.
\end{enumerate}
\end{proposition}
\begin{proof}
(a)\ Consider the linear operator ${\cal Q} := u \in H^1_0(\O) \mapsto q \, u \in L^3(\Omega)$; clearly, $\Lambda_q = H^1_0(\O) \setminus {\cal Q}^{-1}(0)$. By H\"older's inequality and Sobolev's embedding theorem,
${\cal Q}$ is continuous, hence ${\cal Q}^{-1}(0)$ is closed in $H^1_0(\O)$ and $\Lambda_q$ is open.
Moreover, ${\cal Q}^{-1}(0)$ is a proper linear subset of $H^1_0(\O)$, and thus, it  has empty interior;
it follows at once that $\partial \Lambda_q = {\cal Q}^{-1}(0)$. \\
(b)\ Let $\{u_n\} \subset \Lambda_q$ and assume $\dist(u_n,\partial \Lambda_q)\to 0$. Fix $\eps \in (0,\infty)$. Eventually, $\dist(u_n,\partial\Lambda_q)<{\eps}$, hence $\|\nabla(u_n{-}v_n)\|_2 < {\eps}$ for some $v_n \in \partial \Lambda_q$, and   \begin{equation*}
\|q\, u_n\|_3  = \|q\, (u_n{-}v_n)\|_3
 \le \|q\|_6 \, \|u_n{-}v_n\|_6 < c \, \eps \, ,
\end{equation*}
for some $c \in (0,\infty)$. This proves that $\|q \, u_n\|_3 \to 0$.\\
(c)\ Let $S$ be the essential support of $q$, defined as the complement in $\O$ of the largest open set in which $q$ equals zero almost everywhere; note that $|S|>0$. \\
Fix $k \in \N$. Let $A_1, \ldots, A_k$ be pairwise disjoint open subsets of $\O$ that have nonempty intersection with $S$. For every $i\in\{1,\ldots,k\}$, we can choose $u_i$ in ${\mathcal D}(A_i)\cap \Lambda_q$.
(If no such function existed, we would have $q\, u=0$ for every $u\in {\mathcal D}(A_i)$, which implies $q=0$ a.e.~in $A_i$, whence $A_i \subset \O \setminus S$, a contradiction.)
Clearly, $u_1,\ldots,u_k$ are linearly independent elements of $\Lambda_q$.
It follows that $\Lambda_q$ contains spheres of arbitrary dimension, which proves the assertion.
\end{proof}

\begin{remark}
The arguments in the proof of Proposition~\ref{lambda-q} apply to any multiplication operator between Lebesgue spaces and show that the kernel has infinite codimension.
\end{remark}

\section{The constrained functional}\label{J-birth}

In view of the observations at the beginning of Section~\ref{lambda}, the set
\begin{equation*}
Z_q := \bigl\{ (u,\p) \in \Lambda_q \times H^1(\O) \, | \, F'_\p(u,\p)=0 \bigr\}
\end{equation*}
is the graph of the map $\Phi:\Lambda_q \longrightarrow H^1(\O)$ defined by
\begin{equation*}
\Phi(u):= {\cal L}_{qu}(\rho_u) \, .
\end{equation*}
Note that $F''_{\p \p}(u,\p)=-2 \, {\cal A}_{qu}$ for every $(u,\p) \in \Lambda_q \times H^1(\O)$, hence $F''_{\p \p}(u,\p)$ is an isomorphism, by Proposition~\ref{lax-mil}(b); moreover, $F''_{\p u}$ and $F''_{\p \p}$ are continuous in $\Lambda_q \times H^1(\Omega)$. This implies that $\Phi$ is continuously differentiable in $\Lambda_q$.

Constraining the functional $F$ on the set $Z_q$ amounts to
considering the functional $J : \Lambda_q \longrightarrow \R$ defined by
\begin{equation*}
J(u) = F\bigl(u,\Phi(u)\bigr) \, \, .
\end{equation*}
The following assertions are a straightforward consequence of the construction of $J$.

\begin{proposition}\label{CF}~{ }
The functional $J$ is continuously differentiable in $\Lambda_q$. Furthermore,
$(u,\p) \in \Lambda_q \times H^1(\O)$ is a critical point of $F$ if and only if $u$ is a critical point of $J$ and $\p = \Phi(u)$.
\end{proposition}

On account of Proposition~\ref{CF}, nontrivial solutions to Problem~\eqref{SYS-no}-\eqref{BC-no} are in one-to-one correspondence with critical points of $J$ in $\Lambda_q$.

\begin{remark}\label{even}
By the very definition of $\Phi$, we have $\Phi(u)=\Phi(|u|)$ for every $u \in \Lambda_q$; this readily implies $J(u)=J(|u|)$ for every $u \in \Lambda_q$.
\end{remark}

Before investigating further properties of $J$, we note that
\begin{equation*}
\Phi(u) = \eta_u + \xi_u \, ,
\end{equation*}
with $\eta_u := {\cal L}_{qu} ({A}/{|\O|})$ and $\xi_u := -{\cal L}_{qu} ((q \, u)^2 \, \chi)$, for every $u \in \Lambda_q$. By Remark~\ref{equation}, $\eta_u$ and $\xi_u$ satisfy the equations
\begin{align}
- \D \eta_u + (q \, u)^2 \, \eta_u  &=  \dfrac{A}{|\O|} \, , \label{eta-u} \\[2mm]
- \D \xi_u + (q \, u)^2 \, \xi_u &= - \, (q \, u)^2 \, \chi  \, , \label{xi-u}
\end{align}
respectively, with homogeneous Neumann boundary conditions.

\begin{lemma}\label{eta-xi}~{}
\begin{itemize}
\item[{\rm (a)}] For every $u \in \Lambda_q$, $A \, \eta_u \ge 0$ in $\O$.
\item[{\rm (b)}] Let $\gamma \in(0,\infty)$ be such that $\|f-\overline f\|_3 \le \gamma \, \|\nabla f\|_2$ for every $f\in H^1(\O)$. Then:
$\|\nabla \eta_u\|_2 \le \gamma \, \|q\, u\|_3^2 \; | \overline \eta_u |$  for every $u \in \Lambda_q$.
\item[{\rm (c)}] Suppose $A\ne 0$. If $u\in \Lambda_q$ and $\|q \, u\|_3 \to 0$, then $|\overline \eta_u| \to \infty$.
\item[{\rm (d)}] For every $u \in \Lambda_q$, $\|\xi_u \|_{\infty} \le \|\chi\|_\infty$.
\end{itemize}
\end{lemma}

\begin{Proof} (a)\ The assertion is a straightforward consequence of Remark~\ref{max-prin}. \\
(b)\ Fix $u \in \Lambda_q$. Multiplying~\eqref{eta-u} by $\eta_u - \oveta$ yields
\begin{equation*}
\|\nabla \eta_u\|_2^2 + \int_\O (q\, u)^2 \, \eta_u \, (\eta_u - \oveta) \,   dx  = 0 \, ,
\end{equation*}
whence
\begin{equation*}
\begin{split}
\|\nabla \eta_u\|_2^2
& \le  \|\nabla \eta_u\|_2^2 + \int_\O (q\, u)^2 \, (\eta_u-\oveta)^2 \, dx = -  \int_\O (q\, u)^2 \, \oveta \, (\eta_u - \oveta)  \, dx \\
&  \le
 \|q\, u\|_3^2 \; | \oveta | \,  \|\eta_u {-} \oveta\|_3
 \le \gamma \, \|q\, u\|_3^2 \; |\oveta| \,  \|\nabla \eta_u \|_2 \, .
\end{split}
\end{equation*}
(c)\
Integrating~\eqref{eta-u} over $\O$ gives
$\int_\O (q\, u)^2 \, \eta_u \, dx = A$,
whence
\begin{equation}\label{eta1}
|A| \le \int_\O (q\, u)^2 \, |\eta_u| \, dx  \le
\|q \, u\|_3^2 \, \|\eta_u\|_3 \, .
\end{equation}
From~\eqref{eta1} and Part~(b) it follows
\begin{equation*}
\begin{split}
\dfrac{|A|}{\|q \, u\|_3^2} & \le \|\eta_u {-} \oveta\|_3 + \|\oveta\|_3 \\
& \le \gamma \, \|\nabla \eta_u \|_2 + | \overline \eta_u | \, |\O|^{1/3} \le
\bigl(\gamma^2 \, \|q\, u\|_3^2 +  |\O|^{1/3} \bigr) \, | \overline \eta_u | \, ,
\end{split}
\end{equation*}
whence
\begin{equation}\label{eta0}
| \overline \eta_u | \ge \dfrac{|A|}{\|q \, u\|_3^2 \, \bigl(\gamma^2 \, \|q\, u\|_3^2 +  |\O|^{1/3} \bigr)}
\end{equation}
for every $u \in \Lambda_q$. If $u\in \Lambda_q$ and $\|q \, u\|_3 \to 0$,  \eqref{eta0} implies $|\overline \eta_u| \to \infty$. \\
(d)\ Fix $u \in \Lambda_q$.
Let $\tau \in \R$ and define $w_\tau := \xi_u + \tau$; observe that $w_\tau$ solves the equation
$-\Delta w_\tau + (q \, u)^2 \, w_\tau  = (q\, u)^2\, (\tau - \chi)$.
With $\tau = \sup \chi$ (respectively, $\tau = \inf \chi$), Remark~\ref{max-prin} implies $\xi_u \ge -\sup \chi$ (respectively, $\xi_u \le -\inf \chi$) in $\O$. This proves the assertion.
\end{Proof}

Recall that $\chi$ is the unique solution of~\eqref{chi}; by elliptic regularity theory and Sobolev's inequalities, there exists $\kappa \in(0,\infty)$ such that
\begin{equation}\label{chi-infty}
\|\chi\|_{\infty} \le \kappa \, \|\al\|_{1/2} \, .
\end{equation}
Let $\sigma \in (0,\infty)$ be such that $\|u\|_3 \le \sigma \, \|\nabla u\|_2$ for every $u\in H^1_0(\O)$. Let
\begin{equation}\label{delta}
\delta:= \dfrac{1}{\kappa \, \sigma} \, .
\end{equation}

\begin{proposition}\label{PROP-J}
Assume $A\ne 0$ and $\|q\|_6  \, \|\al\|_{1/2} < \delta$. Then:
\begin{itemize}
\item [{\rm (a)}] $J$ is bounded from below and coercive in $\Lambda_q$.
\item [{\rm (b)}] If  $\|q \, u\|_3 \to 0$, then $J(u) \to \infty$.
\item [{\rm (c)}] For $\{u_n\} \subset \Lambda_q$, the sequence $\{J(u_n)\}$ is unbounded if, and only if, either $\{u_n\}$ is unbounded or $\{\|q\, u_n\|_3\}$ is not bounded away from $0$.
    \item [{\rm (d)}] $J$ satisfies the Palais-Smale condition in $\Lambda_q$.
\end{itemize}
\end{proposition}

\begin{proof}
To begin with, let us write the functional $J$ in terms of $u$, $\eta_u$, and $\xi_u$.
To simplify the notation, let $\p_u:= \Phi(u)$. By Remark~\ref{equation}, $\p_u$ solves the homogeneous Neumann problem associated with the equation
\begin{equation*}
- \Delta \p + (q \, u)^2 \, \p =  \dfrac{A}{|\O|} - (q \, u)^2 \, \chi\, \, . \end{equation*}
Then,
\begin{equation*}
\|\nabla \p_u\|_2^2
=  A  \, \overline \p_u  -
\int_\O (q \, u)^2 \, \chi \, \p_u \, dx - \int_\O (q \, u)^2 \, \p_u^2 \, dx  \, ,
\end{equation*}
and thus,
\begin{equation}\label{J-one}
\begin{split}
J(u)  &= F\bigl(u,\p_u\bigr)   \\ &= \|\nabla u\|_2^2 +  \int_\O \bigl(m^2 - q^2 \, \chi^2 \bigr) \, u^2 \, dx -  \int_\O (q \, u)^2 \, \chi \, \p_u \, dx +  A  \, \overline \p_u  .
\end{split}
\end{equation}
Recall that $\p_u=\eta_u+\xi_u$ and observe that
\begin{equation*}
- \int_\O (q \, u)^2 \, \chi \, \eta_u \, dx  =   A \, \ovxi  \, ;
\end{equation*}
this is easily obtained by multiplying Equation~\eqref{eta-u} by $\xi_u$ and Equation~\eqref{xi-u} by $\eta_u$. Substituting into~\eqref{J-one} yields
\begin{equation*}
J(u)   =  \|\nabla u\|_2^2 +  \int_\O \bigl(m^2 - q^2 \, \chi^2 - q^2 \, \chi \,  \xi_u \bigr) \, u^2 \, dx  + 2\, A \, \ovxi +  A \, \oveta
\end{equation*}
for every $u\in \Lambda_q$. \\
(a)\  By~\eqref{chi-infty} and H\"older's inequality,
\begin{equation}\label{co1}
\Bigl|\int_\O q^2 \, \chi^2 \, u^2 \, dx \Bigr|
 \le  \kappa^2 \, \sigma^2 \, \|q\|_6^2 \, \|\al\|_{1/2}^2 \,   \|\nabla u\|_2^2  \, ;
\end{equation}
multiplying~\eqref{xi-u} by $\xi_u$ gives
\begin{equation*}
- \int_\O (q \, u)^2 \, \chi \, \xi_u \, dx = \int_\O \bigl(|\nabla \xi_u|^2 + (q \, u)^2 |\xi_u|^2\bigr) \, dx \ge 0 \,  ;
\end{equation*}
finally, Lemma~\ref{eta-xi}(d) and \eqref{chi-infty} give
\begin{equation}\label{co2}
|\ovxi| \le \kappa \, \|\al\|_{1/2} \, .
\end{equation}
Thus,
\begin{equation}\label{below}
J(u)  \ge  \Bigr[1 - \kappa^2 \, \sigma^2 \, \|q\|_6^2 \, \|\al\|_{1/2}^2 \Bigl] \, \|\nabla u\|_2^2 - 2\, \kappa \, |A| \, \|\al\|_{1/2} +  A \, \oveta \, .
\end{equation}
Note that the quantity within brackets is strictly positive; moreover, $A \, \oveta \ge 0$ by Lemma~\ref{eta-xi}(a). Thus, \eqref{below} implies the desired properties of $J$. \\
(b)\ The assertion readily follows from~\eqref{below} and Lemma~\ref{eta-xi}(c).\\
(c)\ In view of~\eqref{co1}, \eqref{co2}, and the inequality
\begin{equation*}
\Bigl|\int_\O q^2 \, \chi \, \xi_u \, u^2 \, dx \Bigr|
 \le  \kappa^2 \, \sigma^2 \, \|q\|_6^2 \, \|\al\|_{1/2}^2 \,   \|\nabla u\|_2^2  \, ,
\end{equation*}
there exist $c_1, c_2 \in (0,\infty)$ such that
\begin{equation}\label{upper-bound}
J(u)  \le c_1 \, \|\nabla u\|_2^2 + c_2 + |A| \, |\overline{\eta}_u| \quad \hbox{for every $u\in \Lambda_q$.}
\end{equation}
Suppose $\{u_n\}\subset \Lambda_q$ is bounded and $\|q \, u_n\|_3 \ge r$ for every $n$, for some $r\in(0,\infty)$. Up to a subsequence, $\{u_n\}$ has in $L^6(\O)$ a limit $u$. Since $q \, u_n \to q\, u$ in $L^3(\O)$ and $\|q \, u_n\|_3 \ge r$, we deduce $q \, u \ne 0$ and thus, $\eta := {\cal L}_{q u} ({A}/{|\O|})$ is well defined. By Proposition~\ref{lax-mil}(a),  $\eta_{u_n}:= {\cal L}_{qu_n} ({A}/{|\O|})$ converges to $\eta$ in $H^1(\O)$, which implies $|\overline{\eta}_{u_n}| \to |\overline{\eta}|$. Thus, by~\eqref{upper-bound}, the sequence $\{J(u_n)\}$ is bounded. This proves the ``only if'' part of the statement; the ``if'' part easily follows from (a) and (b). \\
(d)\ Suppose that $\{u_n\} \subset \Lambda_q$ is a Palais-Smale sequence, that is, $\{J(u_n)\}$ is bounded and $J'(u_n) \to 0$; we have to show that, up to a subsequence, $\{u_n\}$ converges in $\Lambda_q$.
Since $J$ is coercive, $\{u_n\}$ is bounded in $H^1_0(\O)$; up to a subsequence, it converges weakly to some $u\in H^1_0(\O)$. Observe that
\begin{equation}\label{PS1}
\D u_n  = -  \frac{1}{2} \, J'(u_n) +  m^2 \, u_n  - q^2\, (\eta_{u_n} + \xi_{u_n} + \chi)^2 \,  u_n \, .
\end{equation}
The first and second summands in the right-hand side of~\eqref{PS1} are bounded in $H^{-1}(\O)$.
By~\eqref{below}, the sequence $\{|\overline \eta_{u_n}|\}$ is bounded; Lemma~\ref{eta-xi}(b)
implies that $\{\eta_{u_n}\}$ is bounded in $H^1(\O)$, and thus, in $L^6(\O)$. By Lemma~\ref{eta-xi}(d), $\{\xi_{u_n}+\chi\}$ is bounded in $L^6(\O)$ as well. It follows that $\{(\eta_{u_n} + \xi_{u_n} + \chi)^2\}$ is bounded in $L^3(\O)$, which in turn implies that $\bigl\{q^2 \, (\eta_{u_n} + \xi_{u_n} + \chi)^2 \, u_n \bigr\}$ is bounded in $L^{6/5}(\O)$, and therefore in $H^{-1}(\O)$.
On account of~\eqref{PS1}, the sequence $\{\Delta u_n\}$ is bounded in $H^{-1}(\O)$;
the compactness of the inverse Laplace operator implies that, up to a subsequence, $\{u_n\}$ converges to $u$ in $H^1_0(\O)$. Since $\{J(u_n)\}$ is bounded, Proposition~2.4(b) and Part~(b) imply $u \in \Lambda_q$.
\end{proof}

\begin{remark}
Part~(a) of Proposition~\ref{PROP-J} holds true also if $A =0$.
\end{remark}

\section{Proof of the main results}\label{proof-of-main}

\begin{Proof}[Proof of Theorem~\ref{main}]
On account of the correspondence between critical points of $J$ and nontrivial solutions to Problem~\eqref{SYS-no}-\eqref{BC-no},
it suffices to prove that $J$ has a sequence of critical points $\{u_n\} \subset \Lambda_q$ satisfying (i) and (ii). \\
Suppose $A \ne 0$. With $\delta$ as defined in~\eqref{delta}, assume
$\|q\|_6 \, \|\alpha\|_{1/2}  < \delta$.
By Proposition~\ref{PROP-J}, the functional $J$ is bounded from below, has complete sublevels, and satisfies the Palais-Smale condition in $\Lambda_q$. This readily implies that $J$ attains its minimum at some $u_0 \in \Lambda_q$; by Remark~\ref{even}, we can assume $u_0 \ge 0$ in $\O$. \\
Recall that $\Lambda_q$ has infinite genus, by Proposition~\ref{lambda-q}(c). Thus, Ljusternik-Schnirelmann Theory applies (see~\cite[Corollary 4.1]{szulkin} and \cite[Remark~3.6]{ACZ}) and $J$ has a sequence $\{u_n\}_{n \ge 1}$ of critical points in $\Lambda_q$. Standard arguments show that $J(u_n)\to\infty$ (see~\cite[Chapter 10]{AM}).\\
Let $\{v_{n}\}$ be a bounded subsequence of $\{u_n\}$. By Proposition~\ref{PROP-J}(c), every subsequence of $\{v_{n}\}$ 
has a subsequence $\{v_{k_n}\}$ such that $\|q\, v_{k_n}\|_3\to 0$; this proves that $\|q \, v_n\|_3 \to 0$.
\end{Proof}

\begin{Proof}[Proof of Theorem~\ref{A-zero}]
Given the equivalence between Problem~\eqref{SYS-no}-\eqref{BC-no} and Problem~\eqref{PROB}, it suffices to prove that the latter does not have nontrivial solutions.\\
Assume $\|q\|_6 \, \|\al\|_{1/2} < \delta$.
Let $(u,\p)$ be a solution to~\eqref{PROB} with $A=0$.
Multiplying the first equation in~\eqref{PROB} by $u$ gives
\begin{equation*}\label{u-1}
\|\nabla u\|_2^2  + \int_\O m^2 \, u^2 \, dx - \int_\O (q\, u)^2 \, (\p+\chi)^2 \, dx = 0 \, ,
\end{equation*}
whence
\begin{equation}\label{u-2}
\|\nabla u\|_2^2  + \int_\O \bigl(m^2 - q^2 \, \chi^2) \, u^2 \, dx - \int_\O (q\, u)^2 \, \p^2 \, dx = 2 \int_\O (q\, u)^2 \, \chi\, \p \, dx \, .
\end{equation}
Multiplying the second equation in~\eqref{PROB} by $\p$ gives
\begin{equation}\label{phi-1}
\int_\O (q\, u)^2 \, \chi \, \p \, dx =  - \int_\O (q\, u)^2 \, \p^2 \, dx - \|\nabla \p\|_2^2 \, .
\end{equation}
Substituting~\eqref{phi-1} into~\eqref{u-2}, and taking~\eqref{co1} into account, gives
\begin{equation*}
\begin{split}
0  & = \|\nabla u\|_2^2 + \int_\O \bigl(m^2 - q^2 \, \chi^2 \bigr) \, u^2 \, dx +
\int_\O (q\, u)^2 \, \p^2 \, dx + 2 \, \|\nabla \p\|_2^2 \\
& \ge \Bigr[1 - \kappa^2 \, \sigma^2 \, \|q\|_6^2  \, \|\al\|_{1/2}^2  \Bigl] \, \|\nabla u\|_2^2 \, .
\end{split}
\end{equation*}
Since the quantity between brackets is strictly positive, we get $u=0$.
\end{Proof}

\end{document}